\documentclass{amsart}

\usepackage{bm}
\usepackage{amsmath}
\usepackage{amssymb}
\usepackage[all]{xy}
\newtheorem{thm}{Theorem}[section]
\newtheorem{cor}[thm]{Corollary}
\newtheorem{lem}[thm]{Lemma}
\newtheorem{prop}[thm]{Proposition}
\newtheorem{prp}[thm]{Property}

\newtheorem{defn}[thm]{Definition}
\numberwithin{equation}{section}

\begin{document}

\title{Hyperk\"{a}hler metric and GMN ansatz on focus-focus fibrations}
\author{ Jie \ Zhao }
\address{Jie Zhao, University of Wisconsin-Madison, Madison, WI 53705, USA}
\email{jzhao@math.wisc.edu}

\begin{abstract}
In this paper, we study hyperk\"{a}hler metric and practice GMN's construction of hyperk\"{a}hler metric on focus-focus fibrations. We explicitly compute the action-angel coordinates on the local model of focus-focus fibration, and show its semi-global invariant should be harmonic to admit a compatible holomorphic 2-form. Then we study the canonical semi-flat metric on it.  After the instanton correction, finally we get a reconstruction of the generalized Ooguri-Vafa metric.
\end{abstract}
\maketitle

The well-known Calabi conjecture solved by Yau \cite{34} in 1978 promises that given any elliptic K3 surface there always exists a unique hyperk\"{a}hler metric in each given K\"{a}hler class. After this fundamental existence result, it has been an open problem to write down the explicit expression of such metrics. 

Motivated by the celebrated Strominger-Yau-Zaslow conjecture \cite{31}, nowadays it is a folklore that the hyperk\"{a}hler metrics near large complex limits are approximated by semi-flat metrics with instanton correction from  the holomorphic discs with boundaries on special Lagrangian fibres \cite{10}. The semi-flat metric is wrote down under the special Lagrangian fibration setting in \cite{15}. Later, Gross and Wilson \cite{17} proved that such hyperk\"{a}hler metrics  indeed can be approximated by the semi-flat metrics glued with generalized Ooguri-Vafa metrics around each singular fibre. However, in this procedure the instanton corrections are not included.

Recently, Gaiotto, Moore and Neitzke make a significant  breakthrough and propose a new approach on this problem with the instanton corrctions in their papers \cite{12} \cite{13}. It brings lots of new ingredients into the field, which includes: Kontsevich-Soibelman wall-crossing formula on BPS states or generalized Donaldson-Thomas invariants \cite{21}, and construction of twistor spaces of hyperk\"{a}hler metrics \cite{18} from associated Riemann-Hilbert problems determined by the wall-crossing data.

In this paper, we try to practice the GMN's construction in one of the important cases of completely integrable systems: focus-focus fibration. Follow V\~u Ng\d{o}c's classification result on focus-focus fibration \cite{27}, here we view the local model in \cite{27} as a total neighborhood of an $A_1$ singular fibre in the special Lagrangian fibration of an elliptic K3 surface (up to some symplectomorphism). We adapt GMN's construction as outlined in \cite{26} on the local model, and study the explicit hyperk\"{a}hler metric on it. The paper is organized as follows:

First, we compute the explicit action-angle coordinates on the local model. Then follow Arnold's integration over vanishing cycle technique \cite{1}, we show the semi-global invariant $S$ introduced by V\~u Ng\d{o}c in \cite{27} on the local model should be a harmonic function to admit a compatible holomorphic 2-form.

Second, we study the canonical semi-flat metric on the regular part of local model. We show that under certain conditions on the semi-global invariant $S$, such semi-flat metric will become a real hyperk\"{a}hler metric on the regular part of the fibration. 

Finally, we apply the GMN ansatz to modify the semi-flat metric to a global metric with central fibre completion. In stead of making the modification on the metric directly, we consider the associated twistor space and translate the problem into a Riemann-Hilbert problem on the holomorphic Darboux coordinates of holomorphic 2-forms. Then we follow the GMN integral ansatz to solve the Riemann-Hilbert problem and construct the modified twistor space. From the twistor space, we achieve the final modified metric. It turns out to be the generalized Ooguri-Vafa metric with similar extra harmonic term in the potential function as used in Gross-Wilson's work on hyperk\"{a}hler metric\cite{17}.

\section{focus-focus fibration}

\begin{defn} A Lagrangian fibration $f: (M, \omega) \rightarrow B \in \mathbb{R}^2$ is called a focus-focus fibration if: (1) each fibre is compact; (2) the central fibre $\pi^{-1}(0)$ is the unique singular fibre, which has one $A_1$ singularity, i.e. the central fibre is a pinched torus.
\end{defn}

\begin{defn} Two focus-focus fibrations $f_i: (M_i, \omega_i) \rightarrow B_i$ with $i=1,2$ are called equivalent if there exist subsets $\widetilde{B}_i \subset B_i$ such that we have the following bundle symplectomorphism:
$$\xymatrix{
(M_1|_{\widetilde{B}_1},\omega_1) \ar[rr]^{F} \ar[d]^{f_1} && (M_2|_{\widetilde{B}_2},\omega_2) \ar[d]^{f_2} \\
\widetilde{B}_1 \ar[rr]^{g}  && \widetilde{B}_2
}$$

\end{defn}

For the further construction, we need a quick review of the local model of focus-focus fibration, and also the classification result of focus-focus fibration founded by V\~u Ng\d{o}c \cite{27} here.

\noindent \textbf{Focus-focus Singularity.} Take the space $W= \mathbb{R}^2 \times \mathbb{R}^2$, with the symplectic structure:
\begin{displaymath}
\omega_{can} = dx_1 \wedge dy_1 + dx_2 \wedge dy_2
\end{displaymath}

We consider the following Lagrangian fibration with isolated singularity, namely the focus-focus singularity:
\begin{align*}
 & \pi_{can} : W \longrightarrow \mathbb{R}^2\\
\pi_{can}(x_1,y_1;x_2,y_2) = &(\pi_1,\pi_2) = (x_1y_1+x_2y_2, x_1y_2-x_2y_1)
\end{align*}

If we take an auxiliary complex structure $J_{au}$ on $W$ with:
\begin{displaymath}
z_1=x_1 - iy_1, \quad z_2 = x_2 + iy_2
\end{displaymath}

Then the symplectic structure becomes: $\omega_{can}= Re(d z_1 \wedge d z_2)$, and consequently the fibration simply becomes:
\begin{align*}
& \pi_{can}: W \longrightarrow \mathbb{C} \\
& \pi_{can}(z_1,z_2) = z_1z_2 \\
with: \ & \pi_1= Re(z_1z_2), \quad \pi_2= Im(z_1z_2)
\end{align*}

As a completely integral system, here $\{ \pi_i \} $ induce independent hamitonian flows on $W$. Under the auxiliary complex coordinates, they can be simply written as: 
\begin{align*}
\phi_1^t (z_1,z_2) &= (e^t \cdot z_1, e^{-t} \cdot z_2)\\
\phi_2^t (z_1,z_2) &= (e^{-it} \cdot z_1, e^{it} \cdot z_2)
\end{align*}

\noindent \textbf{Gluing Procedure.} We denote the space of smooth function on $\mathbb{R}^2$ with vanishing value at 0 by $\mathbb{R}[[x,y]]_0$. It will be our classification space. For any $S \in \mathbb{R}[[x,y]]_0$, we denote its partial derivatives by $S_1$ and $S_2$. Then we take two Poincare surfaces in $W$ as follows:
\begin{align*}
\Pi_1 & = \{ (c,1)\mid |c| < \epsilon\} \\
\Pi_2 & = \{ (e^{S_1(c)-iS_2(c)},c \cdot e^{-S_1(c)+iS_2(c)}) \mid |c| < \epsilon\}
\end{align*}

Here $\{ \Pi_i \}$ are smooth surfaces constructed in such a way that for any $c \ne 0$, $\Pi_2$ is the image of $\Pi_1$ by the joint flow of $(\pi_1, \pi_2)$ at the time $(S_1 - \ln |c|, S_2 + \arg(c))$. 

Consider the $S^1$-orbit of $\Pi_i$ under the $\phi_2$ flow, denoted by $\phi_2(\Pi_i)$. We use the symplectomorphism induced by the joint flow to glue collar neighborhoods of $\phi_2(\Pi_i$) inside each torus $\pi^{-1}(c)$, that is:
\begin{align*}
& \psi: \phi_2(\Pi_1) \longrightarrow \phi_2(\Pi_2) \\
\psi (z_1, z_2) = &(e^{S_1(c)-iS_2(c)} \cdot z_2^{-1}, e^{-S_1(c)+iS_2(c)} \cdot z_1 z_2^2 )
\end{align*}

Notice that the gluing is carried out on each Lagrangian fibre. After the gluing procedure, each regular fibre $\pi_{can}^{-1}(c)$ with $ c\ne 0 $ becomes a compact torus, and the central fibre $\pi_{can}^{-1}(0)$ becomes a pinched torus.

Now let us denote the space after the gluing procedure by $(\widetilde{W}, \omega_{can}, S)$. Then we are ready to state the classification result:

\begin{thm}[\cite{27}] The equivalent classes of focus-focus fibration are classified by the local model:
\begin{displaymath}
\pi_{can}: (\widetilde{W}, \omega_{can}, S) \longrightarrow B=\{ c \mid |c| < \epsilon \}
\end{displaymath}
with the classification space $\{ S \mid  S \in \mathbb{R}[[x,y]]_0 \}$.
\end{thm}

\noindent \textbf{Remark.} Since the classification space $\mathbb{R}[[x,y]]_0$ is path connected, by the standard Moser's trick, we will get all the local models with the same base are symplectomorphic to each other.\\

\noindent \textbf{Semi-global Invariant.} The classification data $S$ above is also called the \textit{semi-global invariant} of focus-focus fibration. It has the following geometric interpretation in each focus-focus fibration.

Given a focus-focus fibration $f: (M, \omega) \rightarrow B \in \mathbb{C} \cong \mathbb{R}^2$. Let us take $\{\gamma_{e}, \gamma_{m}\}$ as the generators of $H_1(\pi^{-1}(c))$. If we consider the action integral (central charge) along the 1-cycle:
\begin{displaymath}
z_{\gamma_m}(c) = \frac{1}{2\pi} \int_{\gamma_m} \alpha, \quad z_{\gamma_e}(c) = \frac{1}{2\pi} \int_{\gamma_e} \alpha
\end{displaymath}
where $\alpha$ is any 1-form on some neighbourhood of $\pi^{-1}(c)$ in $\widetilde{W}$ such that $d \alpha = \omega$ (which always exists since $\pi^{-1}(c)$ is Lagrangian). Then the semi-global invariant $S$ can be interpreted as a regularised action integral:
\begin{displaymath}
\qquad \qquad \qquad \quad S(c) = 2\pi \cdot [z_{\gamma_m}(c) - z_{\gamma_m}(0)] + Re(c \ln c - c).                  \qquad \qquad \qquad  (*)
\end{displaymath}

Notice that the classification is purely about the Lagrangian fibration structure. The auxiliary complex structure $J_{aux}$ used above is not necessary compatible with the gluing. In fact, we have the following result:

\begin{lem} The auxiliary complex structure $J_{au}$ is compatible with the gluing if and only if the semi-global invariant $S$ is harmonic.
\end{lem}
\begin{proof} Recall that by definition, we have $S_i(c) = S_i(\frac{z_1z_2+\overline{z_1z_2}}{2}, \frac{z_1z_2-\overline{z_1z_2}}{2i})$. We consider the Cauchy-Riemann equation for the gluing maps, that is:
\begin{displaymath}
\frac{\partial}{\partial \overline{z}_1} (e^{S_1 - i S_2}) = 0, \quad \frac{\partial}{\partial \overline{z}_2} (e^{S_1 - i S_2}) = 0
\end{displaymath}
Such equations can be simplified to: $S_{11} + S_{22} = 0$. Thus we get the proof.

\end{proof}

\section{Action-Angle coordinates}

Now we study the action-angle coordinates of the focus-focus fibration. Let us take a local model $(\widetilde{W}, \omega_{can}, S)$, then denote the punctured disc by $B_0$, and restricted fibration over punctured disc $B_0$ by $\widetilde{W}_0$. We will call $\widetilde{W}_0$ the regular part of the fibration in the later discussion.

Follow the general strategy, we pick a Lagrangian section of the fibration and then use Hamiltonian flows to construct the coordinates.

Recall that from the gluing construction we have the local model given as:
\begin{displaymath}
\widetilde{W} = W / (\Pi_1 \sim \Pi_2)
\end{displaymath}

here the two Poincare surfaces are chosen as:
\begin{align*}
 \Pi_1 &= \{ (c,1)\mid |c| < \epsilon <1 \} \\
 \Pi_2 &= \{ (e^{S_1(c)-iS_2(c)},c \cdot e^{-S_1(c)+iS_2(c)}) \mid |c| < \epsilon < 1\}
\end{align*}

We take a simple Lagrangian section as another initial data:
\begin{displaymath}
 \Gamma(c)=(c,1)
\end{displaymath}

Then we follow the standard procedure in \cite{7} to construct the action-angel coordinates. Use the Hamiltonian flows $\{ \phi_i^t \}$ with $\Gamma(c)$ as the initial level set, we get a parametrization of $\widetilde{W}_0$ as follows:
$$\xymatrix{
B_0 \times \mathbb{R}/L \ar[rr]^{T} \ar[d]^{\pi} && (\widetilde{W}_0,\omega_{can},S) \ar[d]^{\pi_{can}} \\
B_0 \ar[rr]^{id}  && B_0
}$$
\begin{displaymath}
T(c; t_1, t_2) = \phi_1^{-t_1} \circ \phi_2^{-t_2} (c,1) =(c \cdot e^{-t_1 + i \cdot t_2},e^{t_1 - i \cdot t_2} )
\end{displaymath}
Here the period lattice is $L = <(S_1-\ln|c|, S_2 + \arg c),(0, 2\pi)>$ from the above gluing construction.

Notice here $\arg c$ or $\ln c$ is not globally defined function on $B_0$. To clarify the affine coordinates on the base, we need at least two affine charts, with different choice of branches of the $\arg c$ or $\ln c$ function. In our calculation, we will skip this part, and formally use $\arg c$ or $\ln c$ directly if no confusion happens.

\begin{thm} Given the Lagrangian section $\Gamma(c)$ as the initial level set, we have the action-angle coordinates on $\widetilde{W}_0$ as follows:
\begin{align*}
z_{\gamma_m} &= \frac{1}{2\pi} [-\ln|c| \cdot c_1 + \arg c \cdot c_2 + c_1 + S], \quad z_{\gamma_e} = c_2
\\ 
\theta_{\gamma_e} &= \frac{2\pi \cdot t_1}{S_1 - \ln|c|},\qquad \qquad \theta_{\gamma_m} = t_2 -\frac{S_2+ \arg c}{S_1 - \ln |c|} \cdot t_1
\end{align*}
i.e. $\qquad \qquad \qquad T^*(\omega_{can}) =  dz_{\gamma_m} \wedge d\theta_{\gamma_e} + dz_{\gamma_{e}} \wedge d\theta_{\gamma_m}$
\end{thm}

\begin{proof} It is a direct calculation to get the identity:
\begin{displaymath}
T^*(\omega_{can}) = dc_1 \wedge dt_1 + dc_2 \wedge dt_2
\end{displaymath}

Moreover, from the gluing procedure in the local model, we can write down the action integrals directly:
\begin{align*}
z_{\gamma_m} = \frac{1}{2\pi} \left( -\ln|c| \cdot c_1 + \arg c \cdot c_2 + c_1 + S \right), \quad z_{\gamma_e} =c_2
\end{align*}

Consequently, we are able to figure out the frequency data. Recall we have the implicit relations:
\begin{displaymath}
c_1 = c_1(z_{\gamma_m}, z_{\gamma_e}), \quad c_2 = c_2(z_{\gamma_m}, z_{\gamma_e})
\end{displaymath}

Compute the implicit derivatives, then we will get the frequency data:
\begin{align*}
\omega_{1,1} = \frac{\partial c_1}{\partial z_{\gamma_m}} = \frac{2\pi}{S_1 - \ln|c|}, \quad \omega_{1,2} = \frac{\partial c_1}{\partial z_{\gamma_e}} = - \frac{S_2+ \arg c}{S_1 - \ln |c|} 
\end{align*}

Similarly,
\begin{align*}
\omega_{2,1} = \frac{\partial c_2}{\partial z_{\gamma_m}} = 0, \quad \omega_{2,2} = \frac{\partial c_2}{\partial z_{\gamma_e}} = 1 
\end{align*}

Thus we get the angle coordinates:
\begin{align*}
\theta_{\gamma_e} &= \omega_{1,1} \cdot t_1 + \omega_{2,1} \cdot t_2 =  \frac{2\pi \cdot t_1}{S_1 - \ln|c|} \\
\theta_{\gamma_m} &= \omega_{1,2} \cdot t_1 + \omega_{2,2} \cdot t_2 =  t_2 -\frac{S_2+ \arg c}{S_1 - \ln |c|} \cdot t_1
\end{align*}

Finally follow the dynamic identity of the integrable system, we arrive at:
\begin{displaymath}
dc_1 \wedge dt_1 + dc_2 \wedge dt_2 =  dz_{\gamma_m} \wedge d\theta_{\gamma_e} + dz_{\gamma_{e}} \wedge d\theta_{\gamma_m}
\end{displaymath}

That finishes the proof of identities in the lemma.
\end{proof}

Notice that$( dz_{\gamma_m} \wedge d\theta_{\gamma_e} + dz_{\gamma_{e}} \wedge d\theta_{\gamma_m} )$ is invariant under the gluing determined by the period lattice $L$, and also the monodromy transformation of $z_{\gamma_m}$ and $\theta_{\gamma_e}$. Thus the action-angle coordinates above is well defined. Moreover, under the angle coordinates $\{\theta_{\gamma_m}, \theta_{\gamma_e}\}$, the period lattice becomes the standard one: $<(2\pi,0), (0,2\pi)>$.\\

\noindent \textbf{Remark.} Notice that generally the action-angle coordinates is not unique, different choice of Lagrangian section as the zero level set of the Hamiltonian flow may give us different angle coordinates.

\section{Examples}

There are lots of interesting examples of focus-focus fibration studied in different fields. However due to the complicity of elliptic integral generally the action-angle coordinates and thus the semi-global invariant is not easy to calculate. We discuss several cases here.\\

\textit{Example.1.} Spherical pendulum is a famous example equipped with the focus-focus fibration structure. The action integrals and also the semi-global invariant is recently calculated by Dullin in \cite{8}.\\

\textit{Example.2.} The Ooguri-Vafa space $M_{O.V.}$ is also an important case of focus-focus fibration. Geometrically it is a $S^1$ bundle over $\mathbb{R}^2 \times S ^1$ with first chern class $\pm 1$ as constructed in \cite{28}. Follow the Gibbons-Hawking ansatz \cite{14}, we choose the following sympletic form on $M_{O.V.}$ (which is a rescaling by $-2\pi$ of the standard one):
\begin{displaymath}
\omega_0 = -2\pi \cdot [ dc_2 \wedge (\frac{d\theta_m}{2\pi} + A_0) + V_0 d(\frac{\theta_e}{2\pi R}) \wedge dc_1]
\end{displaymath}

Here the Lagrangian fibration is given by:
\begin{align*}
f: & M_{O.V.} \longrightarrow \mathbb{R}^2 \\
f(c_1,c_2 &; \theta_e, \theta_m) = ( c_1,  c_2)
\end{align*}

Recall that the standard Ooguri-Vafa potential \cite{28} is given as:
\begin{displaymath}
V_{O.V.} = \frac{R}{4\pi} \cdot \sum_{n \in \mathbb{Z}}[\frac{1}{\sqrt{ R^2 |c|^2 + (\frac{\theta_e}{2\pi} + n)^2  }}  - \kappa(n) ]
\end{displaymath}
with the regularization terms: $\kappa(0) =0$, and $\kappa(n) = \frac{1}{|n|}$ if $n \ne 0$ for the convergence consideration.

Here we make soem generalization and choose the following potential functions for the symplectic structure $\omega_0$:
\begin{displaymath}
V_0 = \frac{R}{4\pi} \cdot \left( \sum_{n \in \mathbb{Z}}[\frac{1}{\sqrt{ R^2|c|^2 + (\frac{\theta_e}{2\pi} + n)^2  }}  - \kappa(n) ] + 2  S_1(c_1,c_2)\right)
\end{displaymath}
As above, here $S$ is any smooth harmonic function with $S(0)=0$, and $S_i = \frac{\partial}{\partial c_i}S$. It satisfies the following positivity condition along the $\theta_e$-axis:
\begin{displaymath}
S_1(0) > -\min_{\theta_e \in [0,2\pi] } \frac{1}{2}  \sum_{n \in \mathbb{Z}}[\frac{1}{| \frac{\theta_e}{2\pi} + n|  }  - \kappa(n) ]   \qquad \qquad \qquad \qquad  (**)
\end{displaymath}
Then $V$ is still a local positive harmonic function with one singularity at origin. The connection 1-form is given from the standard relation: $dA_0 = * dV_0$. 

\begin{prp}
The action angle coordinates on $(M_{O.V.}, \omega_0)$ with respect to the above Lagrangian fibration can be given by:
\begin{displaymath}
\omega_0 = d z_m \wedge d\tilde{\theta}_e + d z_e \wedge d\tilde{\theta}_m
\end{displaymath}
with:
\begin{align*}
z_m &=  \frac{1}{2\pi} \cdot \left[ Re(c -c \ln c)+ S \right] , \qquad \quad z_e = c_2 \\
\tilde{\theta}_e &= \theta_e + \frac{2\pi \cdot R \sigma}{S_1 - \ln|c|}, \quad \quad \quad \tilde{\theta}_m = -\theta_m - \frac{S_2 + \arg c}{S_1 - \ln|c|} \cdot R \sigma 
\end{align*}
here when $c\ne 0$, the angle correction term is given by:
\begin{displaymath}
\sigma = \int V_0^{inst} d\theta_e = \frac{1}{2\pi}  \sum_{n \ne 0} \frac{1}{i\cdot n} e^{i\cdot n\theta_{e}} K_0(2\pi |nc|) + C
\end{displaymath}
\end{prp}

\begin{proof} It mainly comes from the calculation of the action integrals. Notice that $\omega^{sf}$ and $\omega$ share the same action integrals \cite{4} \cite{12}. It is a special property comes from the Gibbons-Hawking ansatz construction. 

Recall that from Fourier expansion, the semi-flat or zero mode part of the potential is simply given as:
\begin{align*}
V_0^{sf} = - \frac{R}{4\pi} (\ln c + \ln \overline{c} -2  S_1)
\end{align*}

Consequently, we have the semi-flat or zero mode part of the connection 1-form:
\begin{align*}
A_0^{sf} = \frac{i}{8\pi^2} (\ln c - \ln \overline{c} +  2i \cdot S_2) d \theta_e
\end{align*}

Then from direct calculation, we have the action-angle coordinates for the semi-flat part:
\begin{align*}
\omega_0^{sf} &= -2\pi \cdot [ dc_2 \wedge (\frac{d\theta_m}{2\pi} + A^{sf}_0) + V^{sf}_0 d(\frac{\theta_e}{2\pi R}) \wedge dc_1] \\
&= d\left[  \frac{1}{2\pi} \cdot Re(c -c \ln c) + \frac{1}{2\pi} \cdot S \right] \wedge d\theta_e + dc_2 \wedge d(-\theta_m)
\end{align*}

Now let us take the instanton part into account. Recall that the instanton part $\omega_0^{inst} = \omega_0 - \omega_0^{sf}$ is similar determined by by the instanton part of the potential and connection 1-form:
\begin{align*}
V^{inst}_0 &= \frac{R}{2\pi}  \sum_{n \ne 0} e^{i\cdot n\theta_{e}} K_0(2\pi |nc|) \\
A^{inst}_0 &= -\frac{R}{4\pi} \left( \frac{d c}{c} - \frac{d \overline{c}}{\overline{c}} \right) \sum_{n \ne 0} sign(n) \cdot e^{i\cdot n\theta_{e}}|c| K_1(2\pi |nc|)\\
\omega^{inst}_0 &= -2\pi \cdot [ dc_2 \wedge  A^{inst}_0 + V^{inst}_0 d(\frac{\theta_e}{2\pi R}) \wedge dc_1]
\end{align*} 

From property of the Bessel function, we get the following action-angle coordinates for the instanton part:
\begin{align*}
\omega_0^{inst} = d\left[  \frac{1}{2\pi} \cdot Re(c -c \ln c) + \frac{1}{2\pi} \cdot S \right] \wedge d\left( \frac{2\pi \cdot R \sigma}{S_1 - \ln|c|} \right) + dc_2 \wedge  d\left( - \frac{S_2 + \arg c}{S_1 - \ln|c|} \cdot R \sigma \right)
\end{align*}

Add $\omega_0^{sf}$ and $\omega_0^{inst}$ together, we get the whole action-angle coordinates for $\omega_0$.
\end{proof}

\noindent \textbf{Remark.} Notice that $\omega^{sf}_0$ and $\omega_0$ share the same action coordinates but different angle coordinates. In the Ooguri-Vafa space, $\theta_e$ and $\theta_m$ are global coordinates away from the singular point, while the angle coordinates $\tilde{\theta}_e$ and $\tilde{\theta}_m$ generally are only defined on the regular part of the fibration or away from the singular fibre. The above formula in the Ooguri-Vafa case indicates a way to deform the angle coordinates to make them extendable over the singular fibre, which might be helpful in other geometry cases. 

In addition, the formula also shows that in the Ooguri-Vafa case, the instanton correction only contributes to the deformation of the angle coordinates, which comes in the form of an infinite series labeled by wrapping number $n$ as explained in \cite{28}.\\

\textit{Example.3.} Another example of focus-focus fibration comes from the famous special Lagrangian fibration model \cite{16} used in the study of mirror symmetry and wall crossing phenomena:
\begin{displaymath}
\pi: \mathbb{C}^2-\{z_1z_2+1=0\} \longrightarrow \mathbb{R}^2
\end{displaymath}
\begin{displaymath}
\pi(z_1,z_2) = (\ln|1+z_1z_2|, \frac{|z_1|^2-|z_2|^2}{2})
\end{displaymath}
It will be quite interesting to find its action integral and moreover the semi-global invariants.

\section{holomorphic 2-form}

Motivated by study of the hyperk\"{a}hler metric on (local) elliptic K3, here we study holomorphic 2-form or holomorphic symplectic structure on the local model of focus-focus fibration.

Recall that on elliptic K3 surface $X$ with the hyperk\"{a}hler data $(\omega, J, \Omega)$, follow the standard hyperk\"{a}hler rotation, we can always transfer the elliptic fibration structure into a Lagrangian fibration structure with respect to the symplectic form $Re(\Omega)$ or $Im(\Omega)$. Generally we will take local elliptic K3, denoted by $M$, as a total neighborhood of an $A_1$ singular fibre in such fibration. Then equipped with the restricted geometry data, $M$ get a focus-focus fibration structure with respect to the symplectic structure  $Re(\Omega|_{M})$ or $Im(\Omega|_{M})$. 

From classification result \cite{27}, we know such Lagrangian fibration is also equivalent to certain local model of focus-focus fibration. Therefore we would like to study similar geometric structure directly on the local model of focus-focus fibration, which can be viewed as the pull back of the geometric structure from local elliptic K3 through the corresponding bundle symplectomorphism. 

Based on Andreotti's observation \cite{19} about holomorphic 2-form on K3 surface. We make the following definition on the local model:

\begin{defn}
We call a 2-form $\Omega$ on the local model $(\widetilde{W}, \omega_{can}, S)$ is a compatible holomorphic 2-form, if it has the following specialty properties:

1) $\omega_{can} = Re(\Omega)$;

2) $d\Omega=0, \quad \Omega \wedge \Omega =0,  \quad  \Omega \wedge \overline{\Omega} >0 $;

3) the fibration $\pi_{can}: \widetilde{W} \rightarrow B $ becomes an elliptic fibration with respect to the complex structure determined by 2).
\end{defn}
Notice that from Andreotti's argument, in fact any 2-form $\Omega$ with property 2) will determine a unique complex structure $J_0$ such that $\Omega$ becomes a holomorphic 2-form with respect to $J_0$. Thus no ambiguity would happen in our definition. \\

Now let us consider the condition of existence of such 2-form on the local model. Suppose our local model $(\widetilde{W}, \omega_{can}, S)$ admits such a homomorphic 2-form $\Omega$, and the fibration $\pi_{can}: (\widetilde{W},J, \Omega) \rightarrow B$ becomes an elliptic fibration, then similarly as before, we have the action integral (central charge) along the 1-cycle:
\begin{displaymath}
Z_{\gamma_m}(c) = \frac{1}{2\pi} \int_{\gamma_m} \kappa, \quad Z_{\gamma_e}(c) = \frac{1}{2\pi} \int_{\gamma_e} \kappa
\end{displaymath}
where $\kappa$ is any 1-form on some neighbourhood of $\pi^{-1}(c)$ in $\widetilde{W}$ such that $d \kappa = \Omega$ (which always exists since $\pi^{-1}(c)$ is Lagrangian). By construction, we have the simple relation:
\begin{displaymath}
Re(Z_{\gamma_e}) = z_{\gamma_e}, \quad Re(Z_{\gamma_m}) = z_{\gamma_m}
\end{displaymath}

\noindent\textbf{Observations.} The first observation comes from the result of\textit{ integral over vanishing cycles} in singularity theory \cite{1}. Since the local model now is equipped with the structure of elliptic fibration with $A_1$ singularity, we have the property:
\begin{lem} The action integrals $Z_{\gamma_e}, Z_{\gamma_m}$ are holomorphic functions on $B_0$. For the integral over vanishing cycle, we have the local expression:
\begin{displaymath}
Z_{\gamma_m} = f(c)+g(c)\ln(c), \quad \forall c \in B_0
\end{displaymath}
here $f,g$ are local holomorphic functions defined near $c$. Consequently, the action integrals $z_{\gamma_e}$ and $z_{\gamma_m}$ are always harmonic functions on $B_0$.
\end{lem}

Secondly, we have the simple but important identity:
\begin{displaymath}
c-c \cdot \ln c = (-\ln|c| \cdot c_1 + \arg c \cdot c_2 + c_1) + i \cdot (-\ln|c| \cdot c_2 - \arg c \cdot c_1 + c_2)
\end{displaymath}

Thus, we arrive at the following statement about the homomorphic 2-form on the local model:

\begin{cor} If a local model $(\widetilde{W}, \omega_{can}, S)$ admits a compatible holomorphic 2-form $\Omega$ as defined above, then:

 1) the semi-global invariant $S$ is harmonic,

 2) the action integral has the form:
 \begin{displaymath}
 Z_{\gamma_m} = \frac{1}{2\pi} \cdot [c-c \cdot \ln c + (S+ i \cdot \widetilde{S})], \quad Z_{\gamma_e}=c_2-i \cdot c_1=-i \cdot c
 \end{displaymath}
 here $\widetilde{S}$ is conjugate harmonic function of $S$,

 3) the holomorphic 2-form $\Omega$ is determined as:
\begin{displaymath}
\Omega = d Z_{\gamma_m} \wedge d\theta_{\gamma_e}+d Z_{\gamma_e} \wedge d\theta_{\gamma_m}+ i\cdot h(c_1,c_2) dc_1 \wedge dc_2
\end{displaymath}
with the positive definite condition: $S_1 > \ln |c|$. Here $h(c_1,c_2)$ is a smooth function.
\end{cor}

\begin{proof} The first two properties directly comes from the observations made above. We just check the three characteristic properties of two form $\Omega$ here.

1) The closeness of $\Omega$ is given directly.

2) Notice that we have $(S+i \cdot \tilde{S})$ as a homomorphic function, by its Cauchy-Riemann equation, we have the identity:
\begin{displaymath}
d Z_{\gamma_m} \wedge d Z_{\gamma_e} =0
\end{displaymath}
which implies the second property $\Omega \wedge \Omega =0$ through a short calculation.

3) After arrangement, we get the expression:
\begin{align*}
\Omega \wedge \overline{\Omega} 
& = ( d\overline{Z}_{\gamma_e} \wedge d Z_{\gamma_m} + d Z_{\gamma_e} \wedge d \overline{Z}_{\gamma_m} ) \wedge d\theta_{\gamma_e} \wedge d\theta_{\gamma_m} \\
& = \frac{2}{\pi} \cdot (S_1 - \ln |c|) dc_1 \wedge dc_2 \wedge d\theta_{\gamma_m} \wedge d\theta_{\gamma_e}
\end{align*}
therefore we arrive at the positivity condition: $S_1 > \ln |c|$.  The $h$ term appears here since $\theta_{\gamma_e}, \theta_{\gamma_m}$ are not necessary angle coordinates for $Im(\Omega)$ here.

\end{proof}

\noindent \textbf{Remark.} Here the addition term $h(c_1,c_2)$ appears since we just know the fibration is Lagrangian with respect to $Im(\Omega)$. If we further have the section $\Gamma(c)$ is also Lagrangian with respect to $Im(\Omega)$, then the $h(c_1,c_2)$ term will vanish.\\

Generally it is not easy to write down the explicit expression of the complex structure $J_0$ determined by $\Omega$. However from the observation in Lemma 1.4, we have the following result in the special case:
\begin{cor} If a local model $(\widetilde{W}, \omega_{can}, S)$ admits a compatible holomorphic 2-form with $h=0$, then we have the complex structure $J_0$ and the holomorphic 2-form $\Omega$ on $\widetilde{W}$ simply given by:
\begin{displaymath}
J_0= J_{au}, \qquad \Omega = d z_1 \wedge d z_2 = (dc_1 + i \cdot dc_2) \wedge (dt_1 -i \cdot dt_2)
\end{displaymath}

\end{cor}

\noindent \textbf{Remark.} Notice that in this case, $J_0$ and $\Omega$ is well defined on whole $\widetilde{W}$, not just on the regular part $\widetilde{W}_0$. Then the total space $\widetilde{W}$ can be well described as in \cite{38}.

\section{Semi-flat metric}

In this part, we study the canonical semi-flat metric \cite{3} \cite{26} on $\widetilde{W}_0$ constructed by the above action-angle coordinates. From now on, we just focus on the simple case with $h=0$. Notice that in such cases the background complex structure is fixed to be $J_0 = J_{au}$.

\begin{defn}
Given any $R \in \mathbb{R}_{+}$, the canonical semi-flat pseudo-metric on $(\widetilde{W}_0, \omega_{can}, S)$ is given by:
\begin{displaymath}
\omega^{sf} = \pi R \cdot Re(dZ_{\gamma_m} \wedge d\overline{Z}_{\gamma_e})+ \frac{1}{2\pi R} \cdot d\theta_{\gamma_m} \wedge d\theta_{\gamma_e}
\end{displaymath}
\end{defn}

It is easy to check this canonical form $\omega^{sf}$ is compatible with the gluing, and invariant under the monodromy transformation. Moreover, we have the following properties of the canonical form:
\begin{lem}
\begin{displaymath}
\omega^{sf} \wedge \omega^{sf} = \frac{1}{2} \Omega \wedge \overline{\Omega}, \quad \omega^{sf} \wedge \Omega =0, \quad \omega^{sf} \wedge \overline{\Omega} =0
\end{displaymath}
\end{lem}

Thus we get $\omega^{sf}$ indeed a non-degenerate (1,1)-form with respect to $J_0 =J_{au}$. As we know in special geometry, generally $\omega^{sf}$ is just a pseudo-metric on $\widetilde{W}_0$. We still need to check the positivity condition here. From a direct computation, we have the following:

\begin{prop} The canonical form $\omega^{sf}$ gives a hyper-k\"{a}hler metric on $(\widetilde{W}_0, J_0)$ if and only $S_1 > \ln |c|$.
\end{prop}

\begin{proof}

 Notice that we have the complex structure $J_0 = J_{au}$. Thus the induced complex structure on cooridnates $\{c_1,c_2; t_1, t_2\}$ is given by: $$J(dc_1) = dc_2, \ J(dt_1)=-dt_2$$

Consider the canonical form under the $\{ c_1,c_2,t_1,t_2 \}$ coordinates, which is explicitly given as follows:
\begin{align*} 
 \omega^{sf} &=  R (S_1-\ln|c| ) dc_1 \wedge dc_2 + \frac{1}{2\pi R}  d\theta_{\gamma_m} \wedge d\theta_{\gamma_e} \\ 
&= R (S_1-\ln|c| ) dc_1 \wedge dc_2  \\ 
 &\quad - \frac{t_1}{R(S_1 - \ln|c|)^2} [dS_2 + d \arg(c)] \wedge dt_1 + \frac{t_1}{R(S_1 - \ln|c|)^2} [dS_1 - d \ln|c|] \wedge dt_2\\ 
 &\quad + \frac{t_1^2}{R(S_1 - \ln|c|)^3} [dS_2 + d \arg (c)] \wedge  [dS_1 - d \ln|c|]    - \frac{1}{R(S_1 - \ln|c|)} dt_1 \wedge dt_2 
\end{align*}

 For abbreviation, we take some notations here:
 $$ m = S_{11} - \frac{c_1}{|c|^2}, \ n = S_{12} - \frac{c_2}{|c|^2}$$
 
 Then we continue the calculation and get the simplification:
 \begin{align*}  
 \omega^{sf} &= \left[ R (S_1-\ln|c| ) + \frac{t_1^2 (m^2 + n^2)  }{R(S_1 - \ln|c|)^3}  \right] dc_1 \wedge dc_2\\
 & \quad  - \frac{t_1}{R(S_1 - \ln|c|)^2} [n dc_1 -m dc_2] \wedge dt_1 \\
 & \quad  + \frac{t_1}{R(S_1 - \ln|c|)^2} [m dc_1 +n dc_2] \wedge dt_2 \\ 
 & \quad  - \frac{1}{R(S_1 - \ln|c|)} dt_1 \wedge dt_2 
\end{align*}

 By Sylvester's criterion, the positivity condition goes to:
 \begin{align*} 
& S_1 - \ln|c| >0 , \quad \left[  R (S_1-\ln|c| ) + t_1^2 \cdot \frac{m^2 + n^2  }{R(S_1 - \ln|c|)^3}  \right] > 0 \\
 &\left[ R(S_1-\ln|c| ) + t_1^2 \cdot \frac{ m^2 + n^2  }{R(S_1 - \ln|c|)^3}  \right] \cdot \frac{1}{R(S_1 - \ln|c|)}  > [\frac{t_1 n}{R(S_1 - \ln|c|)^2} ]^2 + [\frac{t_1 m}{R(S_1 - \ln|c|)^2} ]^2
 \end{align*}
Previously we already have the condition for $\Omega$ in Corollary 4.3, that is: $S_1 > \ln |c|$. Since $R \in \mathbb{R}_{+}$, thus the final condition goes to:
$
S_1 > \ln |c|.
$
\end{proof}

Therefore, for the semi-flat metric, we still need the same condition: $S_1 > \ln |c|$ as for the holomorphic 2-form $\Omega$, which happens to coincide with the positive asymptotic condition for the Ooguri-Vafa potential at infinity. 

Furthermore, if we pick complex coordinates as $u_1 = c_1 +i c_2, u_2 = t_1 -it_2$, then from the above calculation, we will get a decomposition of the semi-flat metric. In fact, we have:
\begin{align*}
& \pi R \cdot Re(dZ_{\gamma_m} \wedge d\overline{Z}_{\gamma_e}) = R[S_1 - \ln|c|]dc_1 \wedge dc_2 \\
& \frac{1}{2\pi R} \cdot d\theta_{\gamma_m} \wedge d\theta_{\gamma_e} =  i\partial \bar{\partial} \left[ \frac{t_1^2}{R(S_1 - \ln |c|)} \right] = i\partial \bar{\partial} \left[ \frac{S_1-\ln|c|}{4\pi^2 R} \cdot \theta^2_{\gamma_e} \right]
\end{align*}

Notice that $\left( S_1 - \ln|c| \right) $ is a harmonic function on the base $B_0$, thus the first part generally is not a Ricci flat metric on the base, unless it is a flat one. The second part is a pseudo-metric on $\widetilde{W}_0$ since the metric is positive-semidefinite. This decomposition also indicates that $[\omega^{sf}] \ne 0$ in $H^{1,1}(\widetilde{W}_0, J_0)$. \\

\noindent \textbf{Remark.} It is easy to see that in fact we can generalize the canonical semi-flat metric by admitting the parameter $R$ to be a suitable positive functions which is compatible with the gluing and monodromy condition. Similar results as in Lemma 5.2 are still valid, however $\omega^{sf}$ is not be K\"{a}hler anymore.\\

\noindent \textbf{Remark.} Notice that if the $h$ term in the holomorphic form $\Omega$ is non-vanishing, then it is easy to check $\omega \wedge \Omega \ne 0$, and $\omega \wedge \bar{\Omega} \ne 0$, thus the canonical form $\omega^{sf}$ cannot be a (1,1) form or a K\"{a}hler metric. Moreover, in this case, the complex structure determined by $\Omega$ is not $J_{au}$ anymore. It is not easy to write down the explicit expression of the complex structure in general case. Therefore, although we start the construction on the real completely integral system, so far we are just able to carry out further calculation in the complex integrable system, i.e. $h$ term vanishing case. We'd like to explore the general case in further studies.

\section{instanton correction}

Now we consider the instanton correction of the semi-flat metric.

The main strategy here is that we do not make the modification directly on the semi-flat metric, but in stead on its associated twistor space. Then we transfer the correction problem into a Riemann-Hilbert problem of solving the monodromy of the associated holomorphic Darboux coordinates. Finally, we adapt GMN's integral ansatz and read out the metric from the twistor space determined by modified holomorphic Darboux coordinates.

First we consider the twistor space of $\widetilde{W_0}$. By the previous construction of holomorphic 2-form $\Omega$ and K\a"{a}hler form $\omega^{sf}$, we can write the family of holomorphic 2-forms as:
\begin{displaymath}
\varpi^{sf}(\zeta) = -\frac{1}{2 \zeta} \cdot \Omega + \omega^{sf} +\frac{1}{2} \zeta  \cdot \overline{\Omega}, \quad  \zeta \in \mathbb{C}P^1
\end{displaymath}

Here we have an important observation by Gaiotto, Moore and Neitzke. In fact, we can represent the holomorphic 2-forms by holomorphic Darboux coordinates:

\begin{lem} [\cite{12}] The $\mathbb{C}P^1$-family of holomorphic 2-form can be rearranged into the form:
\begin{displaymath}
\varpi^{sf}(\zeta) = \frac{1}{2\pi R} \cdot \frac{d \chi_{\gamma_m}^{sf}}{\chi_{\gamma_m}^{sf}} \wedge \frac{d \chi_{\gamma_e}^{sf}}{\chi_{\gamma_e}^{sf}}
\end{displaymath}
here the holomorphic Darboux coordinates are given by:
\begin{displaymath}
\chi_{\gamma_m}^{sf} = \exp [i \cdot \frac{\pi R}{\zeta}\cdot Z_{\gamma_m} - i \theta_{\gamma_m} - i \cdot \pi R\zeta \cdot \overline{Z}_{\gamma_m}]
\end{displaymath}
\begin{displaymath}
\chi_{\gamma_e}^{sf} = \exp [i \cdot \frac{\pi R}{\zeta}\cdot Z_{\gamma_e} + i \theta_{\gamma_e} - i \cdot \pi R\zeta \cdot \overline{Z}_{\gamma_e}]
\end{displaymath} 
\end{lem}

\noindent \textbf{Remark.}
Recall we have the action-angle coordinates of $\Omega$ given by:
\begin{align*}
Z_{\gamma_m} &= \frac{1}{2\pi} \cdot [c -c \cdot \ln c +(S + i \cdot \widetilde{S})], \qquad \quad Z_{\gamma_e} = -i \cdot c
\\ \theta_{\gamma_e} &= \frac{2\pi \cdot t_1}{S_1 - \ln|c|},\qquad \qquad \qquad \theta_{\gamma_m} = t_2 -\frac{S_2+ \arg c}{S_1 - \ln |c|} \cdot t_1
\end{align*}

According to the monodromy transformation:
\begin{displaymath}
Z_{\gamma_m} \rightarrow Z_{\gamma_m}+ Z_{\gamma_e}, \quad \theta_{\gamma_m} \rightarrow \theta_{\gamma_m} -  \theta_{\gamma_e}
\end{displaymath}
the pairing in the Darboux coordinates is the unique one which induces the complex structure, although with the global monodromy: 
\begin{displaymath}
\chi_{\gamma_m}^{sf} \rightarrow \chi_{\gamma_m}^{sf} \cdot \chi_{\gamma_e}^{sf}
\end{displaymath}

Now comes the main idea of GMN's construction. We should carry out instanton correction on the $\mathbb{C}P^1$-family of holomorphic 2-forms by solving the monodromy issue for the whole family of holomorphic Darboux coordinates simultaneously. 

The basic idea to achieve this is to produce the inverse monodromy to cancel the original one. The main tool we need is the Cauchy-Plemelj-Sokhotskii formula for Riemann-Hilbert problem, that is:
\begin{thm} Take a smooth simple curve $l$ on $\mathbb{C}$, for every $C^{0,\alpha}(l)$ function $\varphi$ on $l$, there exist an unique piecewise holomorphic function on $\mathbb{C}$, which is:

1) continuously extendable from $l_+$ to $\overline{l}_+$ as well as from $l_-$ to $\overline{l}_-$,

2) vanishes for large $|z|$,

3) it has the monodormy: $f_{+}-f_{-}= \varphi$ on $l$,

\noindent Moreover, such function is given by the Cauchy-Plemelj-Sokhotskii formula:
\begin{displaymath}
f(z)= \frac{1}{2\pi i} \int_l \frac{\varphi}{t-z}dt
\end{displaymath}
\end{thm}

Take into account of further compatible conditions for the twistor space \cite{26}, we have GMN's integral ansatz for the holomorphic Darboux coordinates:
\begin{thm} [\cite{12}] The instanton corrected holomorphic Darboux coordinates on $\widetilde{W}_0$ can be given by the following integral formulas:
\begin{align} \notag
\chi_{\gamma_e} &= \chi_{\gamma_e}^{sf} \\ \notag
\chi_{\gamma_m} &= \chi_{\gamma_m}^{sf} \cdot \exp \left[ 
\frac{i}{4\pi}\int_{l_+} \frac{d\zeta'}{\zeta'}\frac{\zeta'+\zeta}{\zeta'-\zeta}\ln(1-\textit{X}_{\gamma_e}(\zeta'))- \frac{i}{4\pi}\int_{l_-} \frac{d\zeta'}{\zeta'}\frac{\zeta'+\zeta}{\zeta'-\zeta}\ln(1-\textit{X}^{-1}_{\gamma_e}(\zeta'))
\right]
\end{align}
Consequently, the instanton corrected holomorphic 2-forms are given by:
\begin{displaymath}
\varpi(\zeta) = \frac{1}{2\pi R} \cdot \frac{d \chi_{\gamma_m}(\zeta)}{\chi_{\gamma_m}(\zeta)} \wedge \frac{d \chi_{\gamma_e}(\zeta)}{\chi_{\gamma_e}(\zeta)}
\end{displaymath}
Here $l_{\pm}$ are the rays connecting $0$ and $\infty$, which are away from the BPS rays: $\{ \zeta  \ |\  Re\frac{c}{\zeta} =0 \}$ determined by holomorphic discs bounding vanishing cycle, otherwise the above integral will diverge.

\end{thm}

Now we adapt this integral formula and compute the new twistor space after instanton correction on our local model.

\begin{cor} After the instanton correction given by GMN ansatz, we get the modified twistor space given as:
\begin{displaymath}
\varpi(\zeta) = \frac{1}{2\pi R} \cdot \xi_m \wedge \xi_e
\end{displaymath}
where:
\begin{align*}
\xi_m &= -i d \theta_{\gamma_m} + 2\pi i \cdot A + \pi i \cdot V \cdot ( \frac{1}{\zeta} d c - \zeta d \overline{c} ) \\
\xi_e &= i d \theta_{\gamma_e} + \pi R \cdot ( \frac{1}{\zeta} d c + \zeta d \overline{c})
\end{align*}
and the potential function here is given by:
\begin{displaymath}
V =  \frac{R}{4\pi} \cdot \left( \sum_{n \in \mathbb{Z}}[\frac{1}{\sqrt{ R^2 |c|^2 + (\frac{\theta_{\gamma_e}}{2\pi} + n)^2  }}  - \kappa(n) ] + 2S_1 \right)
\end{displaymath}
\end{cor}

\begin{proof} The proof of the identity contains the semi-flat part and the instanton part.\\

\textit{Semi-flat part.} For the semi-flat part, we need to check the following identity:
\begin{displaymath}
\varpi^{sf}(\zeta) =  \frac{1}{2\pi R} \cdot \frac{d \chi^{sf}_{\gamma_m}(\zeta)}{\chi^{sf}_{\gamma_m}(\zeta)} \wedge \frac{d \chi_{\gamma_e}(\zeta)}{\chi_{\gamma_e}(\zeta)}
\end{displaymath}
Notice that on the left side, we have:
\begin{align*}
\varpi^{sf}(\zeta) & = \frac{1}{2\pi R} \cdot \xi^{sf}_m \wedge \xi_e \\
& =  \frac{1}{2\pi R} \left[ -i d \theta_{\gamma_m} + 2\pi i \cdot A^{sf} + \pi i \cdot V^{sf} \cdot ( \frac{1}{\zeta} d c - \zeta d \overline{c} ) \right] \wedge  \xi_e \\
& = \frac{1}{2\pi R} \left[ -i d \theta_{\gamma_m} + 2\pi i \cdot A^{sf} + \pi i \cdot V^{sf} \cdot ( \frac{1}{\zeta} d c - \zeta d \overline{c} ) \right] \wedge \frac{d \chi_{\gamma_e}(\zeta)}{\chi_{\gamma_e}(\zeta)}
\end{align*}
here
\begin{align*}
V^{sf} &= - \frac{R}{4\pi} \left( \ln c + \ln \overline{c} -2 S_1   \right) \\
A^{sf} &= \frac{i}{8\pi^2} \left( \ln c - \ln \overline{c} + 2i S_2  \right) d\theta_{\gamma_e}
\end{align*}

Moreover, a direct calculation verifies that:
\begin{displaymath}
\frac{d \chi^{sf}_{\gamma_{m}}(\zeta)}{\chi^{sf}_{\gamma_m}(\zeta)} = 
\left[ -i d\theta_{\gamma_m} + 2\pi i \cdot A^{sf} + \pi i \cdot V^{sf} (\frac{1}{\zeta}dc - \zeta d\overline{c} ) \right] - \frac{i}{4\pi} \left(\ln c - \ln \overline{c} + 2i  S_2 \right) \frac{d \chi_{\gamma_e}(\zeta)}{\chi_{\gamma_e}(\zeta)} 
\end{displaymath}
Thus we get the identity for the semi-flat part.\\

\textit{Instanton part.} For the instanton part, we need to check the following identity:
\begin{displaymath}
\varpi^{inst}(\zeta) =  \frac{1}{2\pi R} \cdot \frac{d \chi^{inst}_{\gamma_m}(\zeta)}{\chi^{inst}_{\gamma_m}(\zeta)} \wedge \frac{d \chi_{\gamma_e}(\zeta)}{\chi_{\gamma_e}(\zeta)}
\end{displaymath}

The whole calculation is similar as given in \cite{12}. For completeness, we outline the main steps here, which will show how the instanton correction will appear from contour integrals. More details and explanations can be found in \cite{12} and \cite{30}.

Notice that on the left side, we have:
\begin{align*}
\varpi^{inst}(\zeta) & = \frac{1}{2\pi R} \cdot \xi^{inst}_m \wedge \xi_e \\
& = \frac{1}{2\pi R} \left[ 2\pi i A^{inst} + \pi i V^{inst} \left( \frac{1}{\zeta}dc - \zeta d\overline{c} \right)   \right]  \wedge \xi_e
\end{align*}
here
\begin{align*}
V^{inst} &= \frac{R}{2\pi}  \sum_{n \ne 0} e^{in\theta_{\gamma_e}} K_0(2\pi R|nc|) \\
A^{inst} &= -\frac{R}{4\pi} \left( \frac{d c}{c} - \frac{d \overline{c}}{\overline{c}} \right) \sum_{n \ne 0} sign(n) \cdot e^{in\theta_{\gamma_e}}|c| K_1(2\pi R|nc|)
\end{align*}

On the right side, first let us take the partial integrals:
\begin{displaymath}
I_{\pm} = - \frac{i}{4\pi} \int_{l_{\pm}} \frac{d\zeta'}{\zeta'} \frac{\zeta'+\zeta}{\zeta'-\zeta} \left[ \frac{\chi_{\gamma_e}(\zeta')^{\pm 1}}{1- \chi_{\gamma_e}(\zeta')^{\pm 1}}\frac{d \chi_{\gamma_e}(\zeta')}{\chi_{\gamma_e}(\zeta')}  \right]
\end{displaymath}

Then the right side of identity goes to:
\begin{displaymath}
\frac{1}{2\pi R} \cdot (I_{+} + I_{-}) \wedge \frac{d \chi_{\gamma_e}(\zeta)}{\chi_{\gamma_e}(\zeta)}
\end{displaymath}

Thus the identity we want to verify can be simplified into:
\begin{displaymath}
 (I_{+} + I_{-}) \wedge \frac{d \chi_{\gamma_e}(\zeta)}{\chi_{\gamma_e}(\zeta)} = \left[ 2\pi i A^{inst} + \pi i V^{inst} \left( \frac{1}{\zeta}dc - \zeta d\overline{c} \right)   \right]  \wedge \xi_e 
\end{displaymath}

Now let us explicitly compute the left side terms. Notice that:
\begin{align*}
I_{\pm} \wedge \frac{d \chi_{\gamma_e}(\zeta)}{\chi_{\gamma_e}(\zeta)} 
& = \frac{i}{4\pi} \int_{l_{\pm}} \frac{d \zeta'}{\zeta'} \left( \frac{\zeta'+\zeta}{\zeta'-\zeta} \cdot \frac{d \chi_{\gamma_e}(\zeta)}{\chi_{\gamma_e}(\zeta)} \wedge \frac{d \chi_{\gamma_e}(\zeta')}{\chi_{\gamma_e}(\zeta')} \right)
\left[\frac{\chi_{\gamma_e}(\zeta')^{\pm 1}}{1 - \chi_{\gamma_e}(\zeta')^{\pm 1}} \right]\\
& = \frac{i}{4\pi} \int_{l_{\pm}} \frac{d \zeta'}{\zeta'} \left( \frac{\zeta'+\zeta}{\zeta'-\zeta} \cdot \frac{d \chi_{\gamma_e}(\zeta)}{\chi_{\gamma_e}(\zeta)} \wedge 
\left[ 
\frac{d \chi_{\gamma_e}(\zeta')}{\chi_{\gamma_e}(\zeta')} 
- \frac{d \chi_{\gamma_e}(\zeta)}{\chi_{\gamma_e}(\zeta)}
\right] \right)
\left[\frac{\chi_{\gamma_e}(\zeta')^{\pm 1}}{1 - \chi_{\gamma_e}(\zeta')^{\pm 1}} \right]\\
& = \frac{i}{4\pi} \int_{l_{\pm}} \frac{d \zeta'}{\zeta'} \left( -\pi R \cdot \frac{d \chi_{\gamma_e}(\zeta)}{\chi_{\gamma_e}(\zeta)} \wedge 
\left[
\left( \frac{1}{\zeta'} + \frac{1}{\zeta} \right) dc - \left(\zeta'+\zeta \right) d\overline{c} 
\right] \right)
\left[\frac{\chi_{\gamma_e}(\zeta')^{\pm 1}}{1 - \chi_{\gamma_e}(\zeta')^{\pm 1}} \right]
\end{align*}

Take the arrangement:
\begin{align*}
L^{+}_1 &= \int_{l_+} \frac{d \zeta'}{\zeta'}\left( \frac{1}{\zeta}dc - \zeta d\overline{c} \right) \left[\frac{\chi_{\gamma_e}(\zeta')}{1 - \chi_{\gamma_e}(\zeta')} \right] \\
L^{-}_1 &= \int_{l_-} \frac{d \zeta'}{\zeta'}\left( \frac{1}{\zeta}dc - \zeta d\overline{c} \right) \left[\frac{\chi_{\gamma_e}(\zeta')^{- 1}}{1 - \chi_{\gamma_e}(\zeta')^{- 1}} \right] \\
L^{+}_2 &= \int_{l_+} \frac{d \zeta'}{\zeta'}\left( \frac{1}{\zeta'}dc - \zeta' d\overline{c} \right) \left[\frac{\chi_{\gamma_e}(\zeta')}{1 - \chi_{\gamma_e}(\zeta')} \right] \\
L^{-}_2 &= \int_{l_-} \frac{d \zeta'}{\zeta'}\left( \frac{1}{\zeta'}dc - \zeta' d\overline{c} \right) \left[\frac{\chi_{\gamma_e}(\zeta')^{- 1}}{1 - \chi_{\gamma_e}(\zeta')^{- 1}} \right] 
\end{align*}

Thus we have the simplification:
\begin{align*}
 (I_{+} + I_{-}) \wedge \frac{d \chi_{\gamma_e}(\zeta)}{\chi_{\gamma_e}(\zeta)} = \frac{iR}{4}(L^{+}_1 + L^{-}_1 + L^{+}_2 + L^{-}_2) \wedge \frac{d \chi_{\gamma_e}(\zeta)}{\chi_{\gamma_e}(\zeta)} 
\end{align*}

Now we are ready to calculate the contour integrals. During the computation, we need some identity for Bessel functions.
Notice that after expanding the geometric series, we get the series labeled by $n$:
\begin{align*}
\int_{l_+} \frac{d \zeta'}{\zeta'} \frac{\chi_{\gamma_e}(\zeta')}{1 - \chi_{\gamma_e}(\zeta')} & = \sum_{n > 0} 2e^{in\theta_{\gamma_e}} K_0(2\pi R |nc|) \\
\int_{l_-} \frac{d \zeta'}{\zeta'} \frac{\chi_{\gamma_e}(\zeta')^{-1}}{1 - \chi_{\gamma_e}(\zeta')^{-1}} & = \sum_{n < 0} 2e^{in\theta_{\gamma_e}} K_0(2\pi R |nc|) 
\end{align*}
Thus, we have the first important identity about the $V^{inst}$ part:
\begin{align*}
\frac{iR}{4} (L^{+}_1 + L^{-}_1) \wedge \frac{d \chi_{\gamma_e}(\zeta)}{\chi_{\gamma_e}(\zeta)} & = \frac{iR}{2} \sum_{n \ne 0} e^{in\theta_e} K_0(2\pi R |nc|) \left( \frac{1}{\zeta}dc - \zeta d\overline{c}    \right) \wedge \frac{d \chi_{\gamma_e}(\zeta)}{\chi_{\gamma_e}(\zeta)} \\
& =\pi i V^{inst} \left( \frac{1}{\zeta}dc - \zeta d\overline{c}    \right) \wedge  \frac{d \chi_{\gamma_e}(\zeta)}{\chi_{\gamma_e}(\zeta)} \\
& = \pi i V^{inst} \left( \frac{1}{\zeta}dc - \zeta d\overline{c}    \right) \wedge \xi_e
\end{align*}
Moreover, we have the identity:
\begin{align*}
\int_{l_+} \frac{d \zeta'}{\zeta'}\zeta' \frac{\chi_{\gamma_e}(\zeta')}{1 - \chi_{\gamma_e}(\zeta')} & = - \sum_{n > 0} 2 \frac{|c|}{\overline{c}} e^{in\theta_{\gamma_e}} K_1(2\pi R |nc|) \\
\int_{l_+} \frac{d \zeta'}{\zeta'}\frac{1}{\zeta'} \frac{\chi_{\gamma_e}(\zeta')}{1 - \chi_{\gamma_e}(\zeta')} & = - \sum_{n > 0} 2 \frac{|c|}{c} e^{in\theta_{\gamma_e}} K_1(2\pi R |nc|) 
\end{align*}
Similarly:
\begin{align*}
\int_{l_-} \frac{d \zeta'}{\zeta'}\zeta' \frac{\chi_{\gamma_e}(\zeta')^{-1}}{1 - \chi_{\gamma_e}(\zeta')^{-1}} & =  \sum_{n < 0} 2 \frac{|c|}{\overline{c}} e^{in\theta_{\gamma_e}} K_1(2\pi R |nc|) \\
\int_{l_-} \frac{d \zeta'}{\zeta'}\frac{1}{\zeta'} \frac{\chi_{\gamma_e}(\zeta')^{-1}}{1 - \chi_{\gamma_e}(\zeta')^{-1}} & =  \sum_{n < 0} 2 \frac{|c|}{c} e^{in\theta_{\gamma_e}} K_1(2\pi R |nc|) 
\end{align*}

Thus $L^{\pm}$ will contribute to the rest part of the identity. We get the second important identity about the $A^{inst}$ part:
\begin{align*}
\frac{iR}{4} (L^{+}_2 + L^{-}_2) \wedge \frac{d \chi_{\gamma_e}(\zeta)}{\chi_{\gamma_{e}}(\zeta)} & = 
\left[ -\frac{iR}{2} \left( \frac{dc}{c} - \frac{d \overline{c}}{\overline{c}}\right)\sum_{n \ne 0} sign(n) \cdot e^{in\theta_{\gamma_e}} |c| K_1(2\pi R|nc|)   \right] \wedge \frac{d \chi_{\gamma_e}(\zeta)}{\chi_{\gamma_e}(\zeta)} \\
& = 2\pi i A^{inst} \wedge \frac{d \chi_{\gamma_e}(\zeta)}{\chi_{\gamma_e}(\zeta)} \\
& = 2\pi i A^{inst} \wedge \xi_e 
\end{align*}

Combine the two important identities about the instanton contribution, finally we finish the proof of the instanton part through:
\begin{align*}
 (I_{+} + I_{-}) \wedge \frac{d \chi_{\gamma_e}(\zeta)}{\chi_{\gamma_e}(\zeta)} & = \frac{iR}{4}(L^{+}_1 + L^{-}_1 + L^{+}_2 + L^{-}_2) \wedge \frac{d \chi_{\gamma_e}(\zeta)}{\chi_{\gamma_e}(\zeta)} \\
 & = \left[ 2\pi i A^{inst} + \pi i V^{inst} \left( \frac{1}{\zeta}dc - \zeta d\overline{c} \right)   \right]  \wedge \xi_e
\end{align*}
\end{proof}

Then follow the standard procedure from twistor space to the metric, we are able to read out the explicit metric. However, here we adapt an alternative way. In fact, from direct comparison to the twistor space $\varpi_{o.v.}(\zeta) $ of Ooguri-Vafa metric with potential function $V$, we can figure out the twistor space $\varpi(\zeta) $ we constructed above is just a rescaling by the constant of $2\pi$, that is $\varpi(\zeta) = 2\pi \varpi_{o.v.}(\zeta) $. 

Notice that because of the orientation issue of the local model, here we need to take a negative sign for $\theta_{\gamma_m}$. Up to this orientation adjustment, we end with the following result:

\begin{thm} Given a harmonic semi-global invariant $S$ with $S_1 > \ln |c|$ and positivity condition $(**)$, the twistor structure determined by $\varpi(\zeta) $ gives a construction of metric $2\pi g$ on $\widetilde{W}_0$, here $g$ is the generalized Ooguri-Vafa metric with potential function $V$ given as above.
\end{thm}

Notice that the metric we directly get from the above construction is just defined on the regular part $\widetilde{W}_0$
of the fibration over $B_0 = \{ 0< |c| < \epsilon \}$. One natural question is how to carry out certain completion of the metric, by adding the central fibre. Here, we adapt the partial completion through Ooguri-Vafa space as follows.\\

\noindent \textbf{Partial completion}. Since now $\widetilde{W}_0$ is equipped with the metric $2\pi g$, by the property of Ooguri-Vafa space, we have a partial metric completion of $\widetilde{W}_0$ by addition a central fibre, which comes from the following isometric embedding:
\begin{align*}
i: (\widetilde{W}_0,2\pi g) &\longrightarrow (M_{O.V.},2\pi g) \\
(z_{\gamma_e}, z_{\gamma_m}, \theta_{\gamma_m}, \theta_{\gamma_e}) &\longmapsto (z_e, z_m, -\theta_{m}, \theta_{e})
\end{align*}
Here the space $M_{O.V.}$, and also the coordinates are explicitly given as in the Example 2.

This partial completion is a little tricky here, since it is not directly from $\widetilde{W}_0$ to $\widetilde{W}$, although we know $M_{O.V.}$ is homeomorphism to $\widetilde{W}$ as the total space of same topological torus fibration. 
The main reason we choose the above approach through embedding $i$ is the following: if originally $\omega_{can}$ comes from a hyperk\"{a}hler metric $g_0$ on the local model, then from construction we can see $ g_{0}|_{ \widetilde{W}_0 } $ it is always different from the metric $2\pi g$ determined by the twistor space $\varpi(\zeta)$.

We make more discussion from the point view of Lagrangian fibration here. Recall that twistor strucutre of $(M_{O.V.},2\pi g)$ is given by $2\pi \varpi_{o.g.}(\zeta)$. With respect to the symplectic structure: $ 2\pi \omega_{o.v.} = -4\pi Re \varpi_{o.g.}(0) $, $M_{O.V.}$ has the Lagrangian fibration structure given by:
\begin{align*}
f: M_{O.V.} & \longrightarrow B \\
f(c_1,c_2; \theta_{m}, \theta_{e}) & = (c_1, c_2)
\end{align*}

As for the local model $\widetilde{W}$, we have the global symplectic structure given by:
$ \omega_{can} = Re(dz_1 \wedge dz_2)$. Its Lagrangian fibration is simply given by:
\begin{align*}
\pi_{can}: \widetilde{W} & \longrightarrow B \\
\pi_{can}(z_1,z_2) & = (c_1,c_2) = \left( Re(z_1z_2), Im(z_1z_2) \right)
\end{align*}

Notice that on the regular part, we have the important relation from the embedding map:
\begin{displaymath}
\omega_{can}|_{\widetilde{W}_0} = i^* ( 2\pi \omega^{sf}_{o.v.} )
\end{displaymath}

Thus we find $i: (\widetilde{W}_0,\omega_{can}) \longrightarrow (M_{O.V.}, \omega_{o.v.})$ is not a symplectic embedding (or equivalently $\omega_{can}|_{\widetilde{W}_0} \ne -2 Re \varpi(0)$), since $\theta_{\gamma_e}, \theta_{\gamma_m}$ are angle coordinates on the left side, however $\theta_e, \theta_m$ are not angle coordinates on the right side, as showed in Property 3.1. The difference essentially can be read out from similar identity as in Property 3.1. The infinite terms of instanton deformation of the angel coordinates here are directly achieved through GMN's integration formula.\\

After the above partial completion, finally we get a re-construction of the Ooguri-Vafa metric (up to a rescaling by $2\pi$ ) with the generalized potential function on a focus-focus fibration through GMN's construction of hyperk\"{a}hler on completely integrable systems. The type of metrics we get coincides with the one used in Gross and Wilson's work on approximation construction of hyperk\"{a}hler metric on elliptic K3. Morever, from point of view of focus-focus fibration, we unfold the geometric meaning of the extra harmonic term in the potential function $V$. In fact, it indicates the semi-global invariant of the integrable system. Such dynamic interpretation also serves as a positive evidence that GMN's project might work finally on the global picture \cite{12} \cite{21}.

\section{Discussion}

At the end, we make some further discussion about the whole construction. Compared to the Gibbons-Hawking ansatz on $S^1$ fibration, the GMN ansatz acts as a construction of hyperk\"{a}hler metric on torus fibration with singular fibres. They require totally different ingredients and technics. However, on the local model, they still have certain relations.\\

\noindent $\bullet$ $A_n$ singularity case:

In our paper, we just study the one singularity case of the Ooguri-Vafa potential, which corresponds to the $A_1$ singularity case. Generally, it is similar to construct the metric on $A_n$ singularity case. Use the GMN ansatz in the local model, we are also able to construct such metrics. For example, as constructed in \cite{27}, we take $n$ copies of the local model with the same semi-global invariant $\{ (W_i, \omega, S) \}$, and denote the Poincare surfaces by $\Gamma_{i,1} \equiv \Gamma_1$, and $\Gamma_{i,2} \equiv \Gamma_2$. Then we make the following sequence of gluing:
\begin{displaymath}
\coprod_{i=1}^n (W_i, \omega, S) / \{ \Gamma_{i,2} \sim \Gamma_{i+1,1} \}
\end{displaymath}
here $\Gamma_{n+1,1} = \Gamma_{1,1}$ and $\Gamma_{n+1,2} = \Gamma_{1,2}$. 

Since the compatible property of the final metric $2\pi g$ on $\widetilde{W}_0$, we just need to take the same metric on each copy and glue them together. After the similar partial completion procedure, we thus get a smooth hyperk\"{a}hler metric in the $A_n$ singularity case.\\

\noindent $\bullet$ $S^1$-symmetry 

As we know, the Ooguri-Vafa metric has the isometric tri-hamiltonian $S^1$ action which is inherited from the Gibbons-Hawking ansatz \cite{2}. However in the GMN ansatz, there is no \textit{a priori} reason such symmetry will appear in the result. Thus it turns out to be interesting in the local model so far we still cannot get a new metric without such symmetry. It may require a further bundle symplectic automorphism (but not holomorphic automorphism) to break such symmetry. \\

\noindent $\bullet$ Removable singularity of GMN ansatz

In the local model, the extension over singular fibre property of the final metric is borrowed from the good property of Ooguri-Vafa metric. However, in the general case, even for two singular fibres case, we have no such auxiliary metric to make the extension work directly. One possible approach comes from the recent development of geometric analysis on the removable singularity of K\"{a}hler-Einstein metrics \cite{5} \cite{6} \cite{37}. GMN ansatz on elliptic K3 may need similar results for hyperk\"{a}hler metric or even twistor space to justify the extension property of the final metric.
\\
\\
\\
\\

\noindent \textbf{Acknowledgements.} 

This work was partially supported by Institute for Basic Science (IBS). The author would like to thank the IBS center for Geometry and Physics in Korea for providing financial supports and excellent environments for research. 

The author appreciate his advisor Prof. Yong-Geun Oh for his suggestion of the topic of focus-focus fibration and helpful discussion on this paper. Moreover, the author appreciate Prof. Andrew Neitzke for the continuous help and discussion during the author's study of the GMN construction.  At last, the author would like to thank Chengjian Yao and Mario Garcia Fernandez for inspiring discussions on the Ooguri-Vafa metric and related topics.

\end{document}